\documentclass[a4paper]{article}

\usepackage[a4paper,
bindingoffset=0.2in,
left=1in,
right=1in,
top=1in,
bottom=1in,
footskip=.25in]{geometry}

\usepackage[fleqn]{amsmath}
\usepackage{amssymb, amsthm, amscd}
\usepackage{url}
\usepackage{graphicx}
\usepackage{algorithm}
\usepackage{algpseudocode}
\usepackage{color}
\usepackage{import}
\usepackage[dvipsnames]{xcolor}

\newcommand{\eg}{e.g.}
\newcommand{\ie}{i.e.}

\newcommand{\HH}{{\mathbb{H}}}
\newcommand{\RR}{{\mathbb{R}}}

\newcommand{\tri}{T}
\newcommand{\M}{M}
\newcommand{\vol}{\mathrm{vol}}
\newcommand{\A}{\mathcal{A}}
\newcommand{\CYRLx}{L}
\newcommand{\tocite}[1]{}
\newcommand{\adddata}[1]{}%
\newcommand{\chs}{CHS}%

\newcommand{\snappy}{\texttt{SnapPy}}
\newcommand{\regina}{\texttt{Regina}}

\newtheorem{theorem}{Theorem}[section]
\newtheorem{lemma}[theorem]{Lemma}

\newtheorem{definition}{Definition}[section]

\newtheorem{remark}{Remark}[section]

\title{Localized geometric moves to compute hyperbolic structures on triangulated 3-manifolds}

\author{Cl\'ement Maria \and Owen Rouill\'e}

\begin{document}
	
	\maketitle

	\begin{abstract}
		A fundamental way to study 3-manifolds is through the geometric lens, one of the most prominent geometries being the hyperbolic one. We focus on the computation of a complete hyperbolic structure on a connected orientable hyperbolic 3-manifold with torus boundaries. This family of 3-manifolds includes the knot complements. 
		
		This computation of a hyperbolic structure requires the resolution of gluing equations on a triangulation of the space, but not all triangulations admit a solution to the equations.
		
		In this paper, we propose a new method to find a triangulation that admits a solution to the gluing equations, using convex optimization and localized combinatorial modifications. It is based on Casson and Rivin’s reformulation of the equations. We provide a novel approach to modify a triangulation and update its geometry, along with experimental results to support the new method.
	\end{abstract}

	\section{Introduction}
	
	A main problem of knot theory is to tell whether two knots are equivalent or distinct. Equivalence between knots is defined by the existence of an isotopy of the ambient space that would turn one knot into the other, \ie, a continuous deformation of the space that preserves the entanglement. 
	
	Isotopies are too difficult to compute in practice, and practitioners use {\it invariants} to tackle the knot equivalence problem. A {\it topological invariant} is a quantity assigned to a presentation of a knot, that is invariant by isotopy. 
	
	An important family of knots are the {\it hyperbolic knots}, which are the knots whose complements admit a {\it complete hyperbolic metric}. They are the subjects of active mathematical research, which motivates the introduction of efficient algorithmic tools to study their geometric properties, and most notably their {\it hyperbolic volume}.  The hyperbolic volume of a hyperbolic knot is a topological invariant which is powerful at distinguishing between non-equivalent knots, is non-trivial to compute, and is at the heart of several deep conjectures in topology~\cite{Murakamivolconj}. 
	
	If it exists, the complete hyperbolic metric on a 3-manifold is unique\cite{Mostowrigid}, and may be combinatorially represented by a \emph{complete hyperbolic structure} (\chs). In order to compute geometric properties (such as volume) of a hyperbolic knot, one triangulates the knot complement, and try to assign hyperbolic shapes to its tetrahedra. If these {\it shapes} verify a set of non-linear constraints called the {\it gluing equations}, they form a \chs\ and encode the complete hyperbolic metric of the space. 
	A major issue with this approach is that a solution to the constraints may not exist on all triangulations of a manifold, even if, as topological object, the manifold can carry a complete hyperbolic metric. Worst, it is not known whether every hyperbolic 3-manifold admits a triangulation on which a solution as \chs\ exists. 
	
	In practice, given an input knot, software can construct a triangulation of the knot complement, and simplify it to have a \chs. But if this fails, the only implemented practical solution, proposed by \snappy\cite{snappy}, is to randomly modify and simplify the triangulation before trying again until a \chs\ is found. 
	
	\begin{table}
		\begin{tabular}{|r|c|c|c|c|c|c|c|c|c|c|c|c|}
			\hline
			& \multicolumn{6}{|c|}{Alternating} & \multicolumn{6}{|c|}{Non-Alternating} \\
			\hline
			\#crossings & 12 & 13 & 14 & 15 & 16 & 17 & 12 & 13 & 14 & 15 & 16 & 17 \\
			\hline
			{\footnotesize\% Failure on first try} & 0.9 & 1.6 & 2.4 & 3.3 & 4.4 & 5.7 & 0.8 & 0.7 & 1.2 & 1.7 & 2.4 & 3.3 \\  
			\hline
			{\scriptsize Expected nb of retriang.} & 2.9 & 3.9 & 6.2 & 7.1 & 9.5 & 15.9 & 2.0 & 3.1 & 6.1 & 11.8 & 9.7 & 13.5 \\
			\hline
		\end{tabular}
		
		\caption{On all $\sim 9.7$ millions prime knots with crossing numbers ranging from 12 to 17, alternating and non-alternating, we indicate ({\it \% Failure on first try}) the percentage of knot complements (after triangulation and simplification) on which \snappy{} fails to compute a \chs\ on first try. We also indicate ({\it Highest expected nb of retriang.}) the highest, over all knots, expected number of random re-triangulations necessary for \snappy{} to succeed finding a \chs. 
		}
		\label{tab:reginacensus}
		
	\end{table}

	\begin{figure}
		\centering
		\includegraphics[scale=0.6]{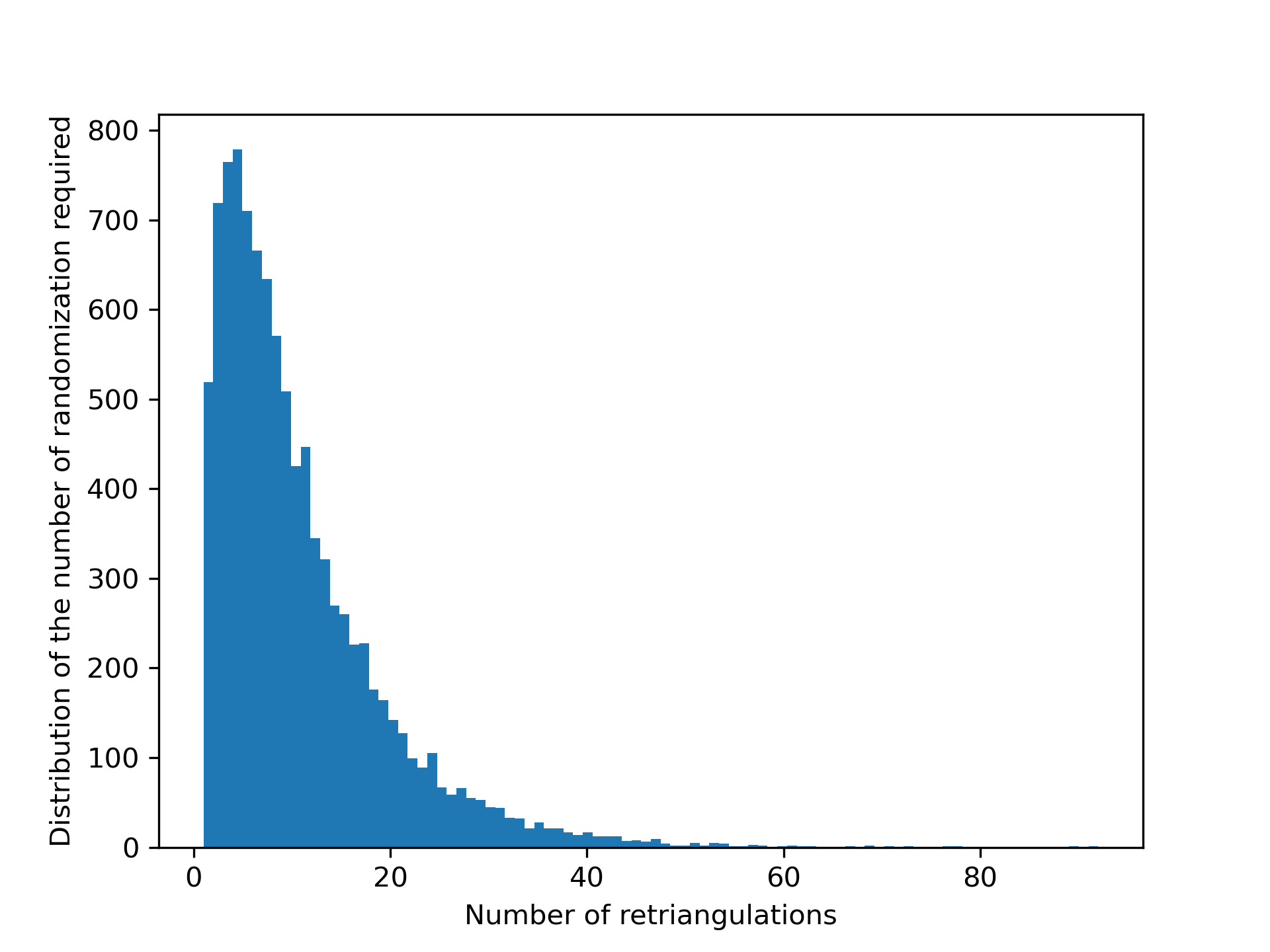}
		\caption{Distribution of the number of randomizations required to find a \chs\ for the complement of the knot ``17nh\_2654001'' of the census, for 10000 tries. The mean is 10.3 randomizations, and the standard deviation is 8.7. The minimum number of re-triangulations is 1, and the maximum 92.}
		\label{fig:17n}
	\end{figure}
	
	Table~\ref{tab:reginacensus} provides data on the search for a \chs\ with \snappy{}, on the $\sim 9.7$ millions prime knots with crossing numbers up to 17. As observed, \snappy{} has a high rate of success in finding a \chs\ after a standard triangulation and simplification of the knot complement. However, this standard construction of a triangulation fails to admit a \chs\ on more than $350$ thousand knots in the census, and the percentage of problematic triangulations needing re-triangulations tends to increase with the number of crossings. Additionally, we observe that some knot complements may require in expectation a high number of random re-triangulations (up to 15.9), and the number of re-triangulations may itself suffer a high variance, as illustrated in Figure~\ref{fig:17n}.
	
	Checking for the existence of a \chs\ requires the resolution of the non-linear gluing equations (with, e.g., Newton optimization method). Reducing the number of re-triangulations is consequently critical for performance of computation in knot theory, most notably for knots on which the state-of-the-art methods implemented in \snappy{} require large numbers of re-triangulations, and even more so when proceeding to very large scale experiments such as the computation of the knot censuses~\cite{burton350} that are of great use to practitioners, where ``difficult'' knots are many.

	\medskip
	
	\noindent
	{\bf Contribution:} This paper introduces a new heuristic based algorithm to improve on the random approach for re-triangulating. The method is inspired by Casson and Rivin's reformulation of the gluing equations~\cite{ThurstonLectures,futer2010angled}: the gluing equations are split into a linear part and a non-linear part, and the resolution reduces to a convex optimization problem on a polytope domain. If the triangulation does not admit a complete structure, the optimization problem will converge on the boundary of the polytope and we exploit this information to introduce new {\em localized moves} to modify combinatorially the triangulation while reusing the partially computed geometry. We introduce necessary background on triangulations and geometry in Section~\ref{sec:background}, and the computation of hyperbolic structures with optimization in Section~\ref{sec:cassonrivin}. We analyze precisely the behavior of the optimization phase on triangulations not admitting a \chs\ in Section~\ref{sec:optim}, and introduce a re-triangulation algorithm in Section~\ref{sec:heuristic} guided by the partially computed geometry of the optimization phase. 
	We illustrate experimentally the interest of the approach in Section~\ref{sec:expe} and propose a hybrid method with \snappy{} in Section~\ref{sec:hybrid}, that outperforms the state-of-the-art.
	
	Note that, in this article we focus on complements of hyperbolic knots. However, the techniques introduced extend to more general hyperbolic 3-manifolds with torus boundaries~\cite{futer2010angled}.

	\section{Background}
	\label{sec:background}
	
	In this article, we focus on knot complements, \ie, non-compact 3-manifolds obtained by removing a closed regular neighborhood of a knot $K$ in the sphere $S^3$.

	\subsection{Generalized and ideal triangulations}
	
	\begin{figure}
		\centering
		\includegraphics[height= 3.2cm]{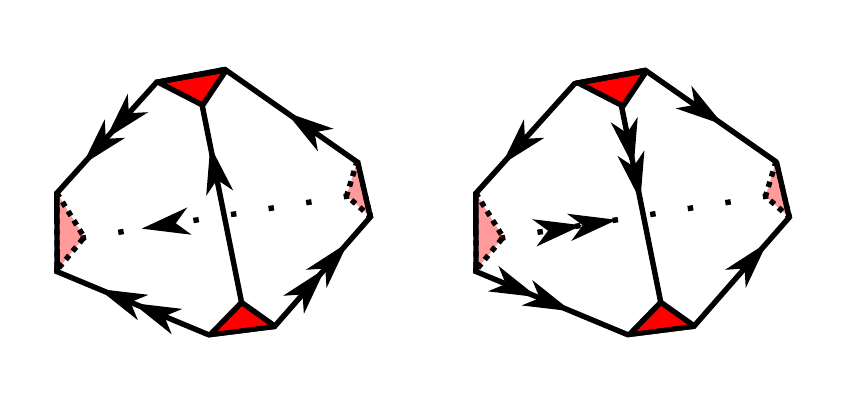}
		\includegraphics[height= 3.2cm]{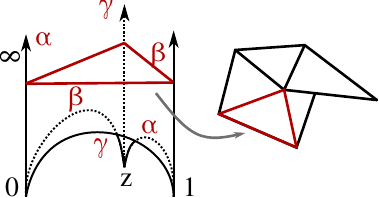}
		\caption{Left: ideal triangulation of the complement of the 8-knot with two tetrahedra. The ideal vertex is truncated, and the red surface gives the torus link of the ideal vertex after gluing of the tetrahedra following the edge identifications. Middle: ideal tetrahedron in the upper half-space model, the dihedral angles are denoted $\alpha$, $\beta$ and $\gamma$, the complex shape parameter $z$ is associated to the edge between $0$ and $\infty$. Right: example of shearing singularity, each triangle corresponds to a ideal tetrahedron seen from above (gray arrow).}
		\label{fig:basetri}
		
	\end{figure}
	
	A \emph{generalized triangulation} $\tri$ is a collection of $n$ abstract tetrahedra whose triangular facets are identified (\eg\ Figure~\ref{fig:basetri} (left)), or \emph{glued}, in pairs. Note that the facets of the same tetrahedron may be glued together, and generalized triangulations are more general than simplicial complexes. The link of a vertex in $\tri$ is the frontier of a closed regular neighborhood, and is itself a closed triangulated surface embedded in the triangulation. If the link of a vertex $v$ is a 2-sphere, we call $v$ an \emph{internal} vertex, otherwise (e.g., if the link if a torus), we call $v$ an ideal vertex. 
	
	A $1$-vertex ideal triangulation is a triangulation with exactly 1 ideal vertex, and no internal vertices. They represent non-compact 3-manifolds, that can be recovered from the 1-vertex ideal triangulation by considering their realization where the vertex has been removed. Every knot complement can be represented by a 1-vertex ideal triangulation, where the link of the ideal vertex gives the frontier of a closed regular neighborhood of the knot. Intuitively, this is a triangulation of the sphere $S^3$ where the knot has been shrunk into a single point, distorting its neighborhood. Such 1-vertex ideal triangulations of a knot complement can be computed in polynomial time~\cite{Jaco0efficient,Hass99UnknotRecNP} from a planar drawing of a knot.

	Any two ideal triangulations of the same 3-manifold can be connected by a sequence of \emph{Pachner moves}~\cite{Pachnermoves}. For 1-vertex ideal triangulations, they consist of the moves 2-3 and 3-2, inverse of each other, pictured in Figure~\ref{fig:pach}.
	
	\begin{figure}
		\centering
		\includegraphics[width = \columnwidth]{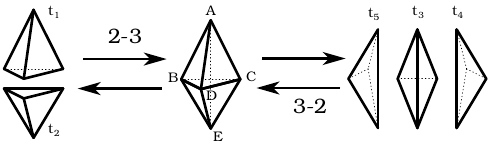}
		\caption{Illustration of the Pachner moves 2-3 and 3-2.}
		\label{fig:pach}
		
	\end{figure}
	
	\subsection{Combinatorial description of hyperbolic geometry}
	\label{sec:hypgeom}

	Certain topological 3-manifolds can be equipped uniformly with a complete hyperbolic metric, which is unique up to isometry. They are called \emph{hyperbolic manifolds}. They include the vast and important family of complements of \emph{hyperbolic knots}.
	
	We use the upper half-space model to represent the hyperbolic space $\HH^3$. This represen\-tation corresponds to $\{(z,r)|z\in\mathbb{C},r\in\mathbb{R}_+^*\}$ with $\partial\overline{\HH^3} = \mathbb{C}\cup\{\infty\}$ (from now on denoted $\partial\HH^3$) consisting of the bottom plane $\mathbb{C} \times \{0\}$, together with the point at infinity, where geodesics are arcs of circles orthogonal to $\partial\HH^3$. This model is \emph{conformal}, \ie, Euclidean angles in the upper half space give the values of the angles in the hyperbolic space. 
	
	An \emph{ideal hyperbolic tetrahedron} is the hyperbolic convex hull of four distinct points of $\partial\HH^3$. These points are the (ideal) vertices of the tetrahedron.
	
	A vertex on $\partial\HH^3$ is a vertex at infinity, thus it is not part of the tetrahedron. An example of ideal tetrahedron is shown in Figure~\ref{fig:basetri} (middle). Up to isometry, the geometric shape of an ideal tetrahedron can be represented by a single complex number:
	
	\begin{definition}[Shape parameter]
		\label{def:shape}
		Given an ideal hyperbolic tetrahedron, there exists an isometry sending three of its vertices to $0$, $1$, and $\infty$, and ensuring that the fourth vertex has positive imaginary part in the complex plane in $\partial\HH^3 = \mathbb{C}\cup\{\infty\}$. The coordinate $z$ of this fourth vertex is the \emph{shape parameter} of the tetrahedron.
	\end{definition}
	\vspace{-0.3cm}
	
	\bigskip
	
	The shape parameter defines and illustrates the shape of an ideal tetrahedron. It depends on which vertices are sent to $0$, $1$, and $\infty$. Other permutations of the vertices give other equivalent shape parameters $z'=\frac{z-1}{z}$ and $z''=\frac{1}{1-z}$ but the underlying tetrahedron is the same. The construction is well defined as isometries of $\HH^3$ are determined by their action on three vertices of $\partial\HH^3$. 
	
	Another way of characterizing the shape of an ideal hyperbolic tetrahedron is to consider its dihedral angles, \ie, the angle formed by two faces meeting on a common edge; see Figure~\ref{fig:basetri} (middle). In an ideal tetrahedron, opposite angles are equal, and the sum of the six dihedral angles is $2\pi$. We denote in the following the dihedral angles of a tetrahedron by a triplet $(\alpha,\beta,\gamma)$ with $(\alpha+\beta+\gamma) = \pi$.

	\section{Angle structures and hyperbolic volume}
	\label{sec:cassonrivin}

	In this section, we introduce notions connected to the computation of complete hyperbolic structures on triangulations, via optimization methods. The approach given in this section is another formulation of Thurston's gluing equations, which are non-linear in the complex shape parameters mentioned above, we refer the reader to~\cite{futer2010angled,ThurstonLectures} for more details.
	
	\subsection{Linear equations and angle structures}

	Let $\tri$ be a 1-vertex ideal triangulation of a knot complement $\M$ with $n$ tetrahedra. 
	Since opposite edges have the same dihedral angles, all possible shapes of the tetrahedra can be represented by a vertex in $\RR^{3n}$. We define an \emph{angle structure}:
	
	\begin{definition}[Angle structures]
		Given an ideal triangulation $\tri$, an \emph{angle structure} is a value assignment to the dihedral angles of the tetrahedra of $\tri$ such that:
		\begin{enumerate}
			\item all the angles are strictly positive;\label{item:pos}
			\item the three dihedral angles $(\alpha, \beta, \gamma)$ of a tetrahedron sum to $\pi$;\label{item:dih}
			\item the dihedral angles around each edge of $\tri$ sum to $2\pi$.\label{item:ed}
		\end{enumerate}
		The set of angle structures on $\tri$ is denoted $\A(\tri)$.
	\end{definition}

	Conditions~\ref{item:pos} and~\ref{item:dih} ensure the angles are in $(0,\pi)$ and that the tetrahedra have the same orientation.
	Condition~\ref{item:ed} is necessary for points on the interior of edges to have a neighborhood isometric to a hyperbolic ball.
	
	Constraints~\ref{item:pos}, \ref{item:dih}, \ref{item:ed} are linear, hence the set $\A(\tri)$ of angle structures of a triangulation is  the relative interior of a polytope in $\RR^{3n}$. It is of dimension $n+|\partial\M|$ where $|\partial\M|$ is the number of cusps of $\M$; in the case of knot complements, there is a single cusp. This polytope satisfies:
	
	\begin{theorem}(Casson; see~\cite{Lackenbydhen})
		Let $\tri$ be an ideal triangulation of $\M$, an orientable 3–manifold with toric cusps. If $\A(\tri)\neq\emptyset$, then $\M$ admits a complete hyperbolic metric.
	\end{theorem}

	Furthermore,~\cite{futer2010angled} describes precisely a generating family for the tangent space of $\A(\tri)$ that can be computed in polynomial time.
	
	\subsection{Maximizing the hyperbolic volume} The hyperbolic volume of an ideal hyperbolic tetrahedron with dihedral angles $(\alpha,\beta,\gamma)\in(0,\pi)^3$ is given by the function $\vol$:
	\[ \vol(\alpha,\beta,\gamma) = \text{\CYRLx}(\alpha)+\text{\CYRLx}(\beta)+\text{\CYRLx}(\gamma)\]
	where \CYRLx\ is the Lobachevsky function $\text{\CYRLx}(x)=-\int_{0}^{x} \log|2\sin t| \, \mathrm{d}t$. The volume functional can be extended to the whole polytope $\A(\tri)$ by summing the volumes of the hyperbolic tetrahedra. The following result is due independently to Casson and Rivin.
	
	\begin{theorem}[Casson, Rivin\cite{Rivin1994EuclideanSO}]
		Let $\tri$ be an ideal triangulation with $n$ tetrahedra of $\M$, an orientable 3–manifold with boundary consisting of tori. Then a point $p\in \A(\tri)\subset\RR^{3n}$ corresponds to a complete hyperbolic metric on the interior of $\M$ if and only if $p$ is a critical point of the function $\vol$. 
		
		Additionally, the volume functional is concave on $\A(\tri)$ and the maximum can be computed via convex optimization methods.
	\end{theorem} 
	
	We can finally define \chs\ that represent combinatorially complete hyperbolic metrics:
	
	\begin{definition}[Complete Hyperbolic Structure]
		A Complete Hyperbolic Structure (\chs) is a triangulation $\tri$ equipped with an angle structure corresponding to the maximum of $\vol$ over $\A(\tri)$.
	\end{definition}

	\begin{remark}
		Angle structures are not \chs\ as they do not prevent shearing singularities, see Figure~\ref{fig:basetri} (right), and consequently may represent \emph{non-complete} hyperbolic metrics.
	\end{remark}

\section{Behavior of the optimization}
\label{sec:optim}
Exploiting the formalism of the previous section, one can design an algorithm~\cite{futer2010angled} to find and compute a \chs\ by first finding a point in the polytope $\A(\tri)$, then computing a basis of the tangent space of $\A(\tri)$, and then maximizing the (concave) hyperbolic volume functional on this subspace. 
If $\A(\tri) \neq \emptyset$ and the procedure finds the point maximizing the volume in the inside of the polytope, then the triangulation admits a complete hyperbolic structure represented by this point. Finally, if $\A(\tri) \neq \emptyset$ but the maximum of the volume functional is on the boundary, then one needs to re-triangulate the manifold in order to search for a triangulation admitting a \chs.

In this section, we study the outcome of the hyperbolic volume maximization on the space of angle structures. The result is given by the following lemma, which also appeared independently in~\cite{Nimershiem21}:

\begin{lemma}
	Let $\tri$ be an ideal triangulation of $\M$, a non-compact orientable 3–manifold with toric cusps. Let $p$ be the point maximizing $\vol$ over $\overline{\A(\tri)}$ (the topological closure of $\A(\tri)$), then at $p$, if a tetrahedron has an angle equal to $0$, then all other angles of the tetrahedron are in $\{0,\pi\}$.
	\label{lem:flat}
\end{lemma}
In other words, either the maximization succeeds in $\A(\tri)$, or there is at least one tetrahedron with angles $(0,0,\pi)$ in $p$. Such tetrahedron is called \emph{flat}. It is not possible to have a tetrahedron with angles $(0,a,b)$, $(a,b)\in(0,\pi)$ in $p$.

\begin{proof}
	
	Let $\vec{p}\in\A(\tri)$ be an angle structure over $\tri$. Let $\vec{w}\in \mathbb{R}^{3n}$ be a vector tangent to $\A(\tri)$, then:

	\[
	\frac{\partial \vol(\vec{p})}{\partial\vec{w}} = \sum_{i=1}^{3n} -w_i \log\sin(p_i).
	\]
	
	Note that because $\vec{w}$ is tangent to $\A(\tri)$, its restriction to a single tetrahedron is three elements whose sum is equal to zero. Let us assume the maximum of the volume is reached on some point $\hat{p}$. Let us assume a tetrahedron $t$ has angles $(0,a,b)$ at $\hat{p}$, $0<a\leq b$, let $\hat{w}$ be a vector tangent to $\A(\tri)$ pointing towards the interior of the polytope with respect to $\hat{p}$ (and thus $w_1>0$) and $\hat{w}_t = (w_1,w_2,w_3)$ be its restriction to $t$. Then: 
	
	\[
	\lim_{\epsilon\rightarrow 0^+} \frac{\partial \vol(t+\epsilon\hat{w}_t)}{\partial\hat{w}_t}= -w_2 \log\sin(a) - w_3 \log\sin(b) + \lim_{\epsilon\rightarrow 0^+}-w_1 \log\sin(\epsilon w_1) = +\infty .
	\]
	
	Now, with the same notations, let us assume $t$ has angles $(0,0,\pi)$ at $\hat{p}$, then using $\log(\sin(x))=\log(x) + O_{x\rightarrow 0}(x^2)$:

	\begin{align*}
		\lim_{\epsilon\rightarrow 0^+} & \frac{\partial \vol(t+\epsilon\hat{w}_t)}{\partial\hat{w}_t}\\
		& =  -w_1\log\sin(\epsilon w_1)-w_2\log\sin(\epsilon w_2)+(w_1+w_2)\log\sin(\pi - \epsilon (w_1+w_2))\\
		& =-w_1\log(\epsilon w_1)-w_2\log(\epsilon w_2)+(w_1+w_2)\log(\epsilon (w_1+w_2)) + O_{\epsilon\rightarrow 0}(\epsilon^2) \\
		& =-w_1\log(w_1)-w_2\log(w_2)+(w_1+w_2)\log(w_1+w_2) + O_{\epsilon\rightarrow 0}(\epsilon^2)\\
		& >-\infty
	\end{align*}

	Since for all tetrahedra with non zero angles or angles of the form $(0,0,\pi)$, the derivative of the volume is bounded, the existence of a tetrahedra with exactly one angle set to zero contradicts the maximality of the volume at that point. Indeed, when getting away from this point, there exists a small neighborhood where the derivative of the volume can be arbitrarily high.

\end{proof}

We introduce in the next section an algorithm that performs localized combinatorial modifications on a triangulation equipped with an angle structure, in order to get rid of flat tetrahedra.

\section{Localized combinatorial modifications of triangulations}
\label{sec:heuristic}
According to Lemma~\ref{lem:flat}, the volume maximization either leads to a solution to the gluing equations, or to flat tetrahedra. In this section, we discuss a method to get rid of flat tetrahedra by combinatorial modifications of the triangulation, while attempting to maintain the value of the volume functional. This would lead to a triangulation admitting an angle structure with larger volume than the previous one, allowing to resume the maximization. In order to maintain the value of the volume functional, we introduce {\em geometric Pachner moves}.

\subsection{Geometric Pachner moves}

We perform Pachner moves on the triangulation that preserves the partial geometric data computed. More precisely, we define:

\begin{definition}[Geometric Pachner move]
	A \emph{geometric Pachner move} in a triangulation $\tri$ with angle structure is a Pachner move in $\tri$ such that the resulting triangulation admits an angle structure  with identical dihedral angles for the tetrahedra not involved in the move.
\end{definition}

To check if a geometric Pachner move can be done, it must be valid combinatorially, and it should be possible to assign dihedral angles to the new tetrahedra without altering the rest of the angle structure. The combinatorial conditions are simple~\cite{Pachnermoves}: for the 2-3 moves (resp. 3-2 move) in Figure~\ref{fig:pachlem}, the tetrahedra sharing the common triangle $BDC$ (resp. common edge $AE$) must be distinct. The conditions on the angle structures are given by the following lemma:

\begin{lemma}
	\label{lem:geom}
	Given a triangulation with angle structure $\tri$, if the combinatorial conditions given above are satisfied:
	\begin{itemize}
		\item a 3-2 move is always geometric;
		\item a 2-3 move is geometric if and only if the sum of the two dihedral angles around each edge of the common face is smaller than $\pi$.
	\end{itemize}
	Furthermore the new angles can be computed using formulas.
\end{lemma}  

\begin{figure}
	\centering
	\includegraphics[width = \columnwidth]{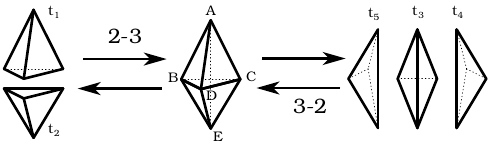}
	\caption{Pachner moves 2-3 and 3-2 with vertices corresponding to the notations of the lemmas of this section.}
	\label{fig:pachlem}
	
\end{figure}

The idea to prove the lemma is to study the system of equations given by the assumption that the angle structure is not altered outside the move, and we break the lemma into three: Lemma~\ref{lem:32geom} concerns the 3-2 move, Lemma~\ref{lem:23geom} concerns the 2-3, and Lemma~\ref{lem:anglevalues} concerns the values of the angles. 

With the notations of Figure~\ref{fig:pachlem}, let $t_1=ABCD$ and $t_2=BCDE$ be two distinct tetrahedra glued along their common face, and $t_3 = ABCE$, $t_4 = ACDE$ and $t_5 = ADBE$ three distinct tetrahedra glued around their common edge. Assuming there is an underlying angle structure, let us denote $t(e)$ the value of the dihedral angle around the edge $e$ in the tetrahedron $t$. Passing from $(t_1,t_2)$ to $(t_3,t_4,t_5)$, \ie\ doing a 3-2 move, or the opposite without altering the rest of the angle structure leads to the following equations:

\begin{equation}
	t_1(BC)+t_2(BC) = t_3(BC),\ t_1(CD)+t_2(CD) = t_4(CD),\ t_1(DB)+t_2(DB) = t_5(DB)
	\label{eqn:pi}
\end{equation}

\begin{equation}
	t_1(AB) = t_3(AB)+t_5(AB),\ t_1(AC) = t_3(AC)+t_4(AC),\ t_1(AD) = t_4(AD)+t_5(AD)
	\label{eqn:t1}
\end{equation}

\begin{equation}
	t_2(EB) = t_3(EB)+t_5(EB),\ t_2(EC) = t_3(EC)+t_4(EC),\ t_2(ED) = t_4(ED)+t_5(ED)
	\label{eqn:t2}
\end{equation}

The constants, the angles that are fixed in the angle structure, and the unknowns, the angles we want to compute, depends on which move is studied.
By equality of the dihedral angles, $t(UV) = t(XY)$ for $UVXY$ vertices of $t$. This equality is enforced by the framework of the angle structures: for the constraints system, $t(UV)$ and $t(XY)$ both represent the same variables.

\subsubsection{The 3-2 move}

\begin{lemma}
	Let $\tri$ be a triangulation with angle structure, let $t_3 = ABCE$, $t_4 = ACDE$ and $t_5 = ADBE$ be three distinct tetrahedra glued around an edge. 
	A 3-2 geometric Pachner move can always be performed at $AE$.
	\label{lem:32geom}
\end{lemma}
\begin{proof}
	The 3-2 move is combinatorially valid as the resulting tetrahedra and their gluings are well defined. Let us have $t_1=ABCD$ and $t_2=BCDE$.
	
	The system composed of Equations~\ref{eqn:pi}, \ref{eqn:t1} and \ref{eqn:t2} is linear with 6 unknowns: the dihedral angles of $t_1$ and $t_2$. We need to check if this system admits a solution where the angles are in $(0,\pi)$. Note that as $t_3(AE)$, $t_4(AE)$, and $t_5(AE)$ are the angles around the edge $AE$, their sum is $2\pi$.
	
	{\setlength{\mathindent}{0cm}\begin{tabular}{rllll}
			$t_1(AB) =$ & $ t_3(AB)+t_5(AB)$ & $=$ & $ 2\pi - (t_3(AE) + t_3(AC) + t_5(AE) + t_5(AD))$\\
			$<$& $2\pi - (t_3(AE) + t_5(AE))$ & $=$ & $t_4(AE)$\\
			$<$& $\pi$ & &
	\end{tabular}}
	
	\begin{tabular}{r}
		$t_1(AB) = t_3(AB)+t_5(AB)>0$
	\end{tabular}
	
	Thus, if the system admits a solution, the dihedral angles of $t_1$ and $t_2$ are in $(0,\pi)$ by symmetry.
	
	Now let us check if all the equations can be satisfied. From Equation~\ref{eqn:t1} and \ref{eqn:t2}:
	\[t_1(AC) + t_2(EC) = t_3(AC)+t_4(AC)+t_3(EC)+t_4(EC) = 2\pi - (t_3(AE)+t_4(AE)) = t_5(AE)\]
	Thus the equations of Equation~\ref{eqn:pi} are consequences of the other two. The system is linear with 6 independent equations and 6 unknowns, it admits a solution.
	
	Finally we have:
	
	{\setlength{\mathindent}{0cm}\begin{tabular}{rllll}
			$t_1(AB)+t_1(AC)+t_1(AD)=$& $t_3(AB)+t_5(AB) + t_3(AC)+t_4(AC) + t_4(AD)+t_5(AD)$ \\
			$=$ &  $3\pi-(t_3(AE)+t_4(AE)+t_5(AE))$ \\
			$=$ &$\pi$
	\end{tabular}}

	Thus the dihedral angles of $t_1$ sum to $\pi$, the same goes for $t_2$ by symmetry. And the move is always valid as long as we started with an angle structure.  
\end{proof}

\subsubsection{The 2-3 move}

\begin{lemma}
	Let $\tri$ be a triangulation with angle structure, let $t_1=ABCD$ and $t_2=BCDE$ be two distinct tetrahedra glued along a face.
	A 2-3 geometric Pachner move can be performed at $BCD$ if and only if $\forall e\in\{BC,CD,DB\},\ t_1(e)+t_2(e)<\pi$.
	\label{lem:23geom}
\end{lemma}

\begin{proof}
	The 2-3 move is combinatorially valid as the resulting tetrahedra and their gluings are well defined. Let us have $t_3 = ABCE$, $t_4 = ACDE$ and $t_5 = ADBE$.
	
	From Equations~\ref{eqn:pi} we have that if $\exists e\in\{BC,CD,DB\},\ t_1(e)+t_2(e)\geq\pi$ then one of the resulting tetrahedra will have a dihedral angle larger than $\pi$, preventing the validity of the move. Let us assume it is not the case.
	
	The system composed of Equations~\ref{eqn:pi}, \ref{eqn:t1} and \ref{eqn:t2} is linear with 9 unknowns: the dihedral angles of $t_3$, $t_4$ and $t_5$. From Equations~\ref{eqn:pi}, the values of $t_3(AE)$, $t_4(AE)$, and $t_5(AE)$ are fixed and we have:
	{\setlength{\mathindent}{0cm}
		\begin{multline*}
			t_3(AE) + t_4(AE) + t_5(AE) \\= t_1(BC)+t_2(BC)+t_1(CD)+t_2(CD)+t_1(DB)+t_2(DB) = 2\pi
	\end{multline*}}
	Thus the constraint around the edge $AE$ is satisfied.
	
	Summing the equations of Equation~\ref{eqn:t1} on the one hand, and the equations of Equation~\ref{eqn:t2} on the other hand both gives:
	\[\pi = t_3(AB)+t_5(AB)+t_3(AC)+t_4(AC)+t_4(AD)+t_5(AD)\]
	
	We also have that the sum of the dihedral angles of $t_3$ is $\pi$ (and by symmetry it is true for $t_4$ and $t_5$):
	
	{\setlength{\mathindent}{0cm}
		\begin{align*}
			t_3(AB)+t_3(AC)+t_3(AE) & = t_1(BC)+t_2(BC)+t_1(AB)+t_1(AC) - t_5(AB)-t_4(AC) \\
			& = t_1(BC)+t_1(AB)+t_1(AC)+t_2(BC)-t_2(ED) \\
			& = \pi
	\end{align*}}

	Now, let us note that fixing any unknown fixes the values of the others, thus the system has exactly one degree of freedom: if $(t_3(AB),t_3(AC),t_4(AC),t_4(AD),t_5(AD),t_5(AB))$ is a solution to the system, then for any $\lambda\in\RR$, $(t_3(AB)+\lambda,t_3(AC)-\lambda,t_4(AC)+\lambda,t_4(AD)-\lambda,t_5(AD)+\lambda,t_5(AB)-\lambda)$ is also a solution. This is discussed in Remark~\ref{rem:23holo}. We denote by $\Lambda_+$ the set of dihedral angle for which $\lambda$ is added, and by $\Lambda_-$ its complement.
	
	The last step of the proof is to show that there exists a choice of $\lambda$ such that all the dihedral angles are positive, given that the dihedral angles of the tetrahedra sum to $\pi$ is enough to have them in $(0,\pi)$. 
	
	By Lemma~\ref{lem:aux23geom}, there is always a choice of dihedral angle such that fixing it to 0 ensures that the elements of $\Lambda_+$ are positive and those of $\Lambda_-$ are non negative (or $\Lambda_-$ and $\Lambda_+$ respectively, depending on whether the selected angle is in $\Lambda_+$ or $\Lambda_-$). If the elements of $\Lambda_+$ are positive, then picking a sufficiently small negative $\lambda$ to modify the dihedral angles will maintain the elements of $\Lambda_+$ positive, and will turn positive those of $\Lambda_-$, in the other case, in the same way a small positive $\lambda$ will ensure the positivity of all the angles.
	
	Thus there exists an angle structure corresponding to and coherent with the result of the Pachner move.
	
\end{proof}

\begin{lemma}
	With the notations and the assumption of the proof of Lemma~\ref{lem:23geom}, there is always a choice of dihedral angle $x$ in  $\{t_3(AB),t_3(AC),t_4(AC),t_4(AD),t_5(AB),t_5(AD)\}$ such that fixing $x=0$ leads to the elements of $\Lambda_+$ being positive and those of $\Lambda_-$ being non negative (or $\Lambda_-$ and $\Lambda_+$ respectively).
	\label{lem:aux23geom}
\end{lemma}

\begin{proof}
	Let us assume the Lemma is false.
	
	If $t_3(AC)$ is set to $0$, then $$t_3(AB)=\pi-t_3(AE)>0,$$ $$t_5(EB)=t_2(EB)>0$$ $$t_4(ED)=t_1(AC)>0.$$ Furthermore $$t_4(AD) = \pi - (t_4(AC) + t_4(AE)) = t_1(AD)-t_2(EB)$$ and $$t_5(AB)=t_2(ED)-t_1(AC).$$ Note that the elements of $\Lambda_+$ are positive and the elements of $\Lambda_-$ are non-negative if $t_4(AD)$ and $t_5(AB)$ are non-negative. Let us assume $t_1(AD)-t_2(EB)<0$.
	
	If $t_4(AD)$ is set to $0$, like previously the elements of $\Lambda_+$ are positive, and $$t_3(AC)=t_2(EB)-t_1(AD)>0$$ $$t_5(AB)=t_1(AB)-t_2(EC).$$ If the lemma is false, then $t_1(AB)-t_2(EC)<0$.
	
	If $t_5(AB)$ is set to $0$, like previously the elements of $\Lambda_+$ are positive, and $$t_3(AC)=t_1(AC)-t_2(ED)$$ $$t_4(AD)=t_2(EC)-t_1(AB)<0$$. If the lemma is false, then $t_1(AC)-t_2(ED)<0$.
	
	The previous deductions give $$t_1(AB)+t_1(AC)+t_1(AD)<t_2(EB)+t_2(EC)+t_2(ED),$$ which is absurd as these two sums are equal to $\pi$. The assumption that $t_1(AD)-t_2(EB)<0$ must be false, and by symmetry $t_2(ED)-t_1(AC)<0$ will be false too. 
	
	This is absurd and the Lemma is true.
	
\end{proof}

\begin{remark}[Non-conservation of the volume]
	\label{rem:23holo}
	In Lemma~\ref{lem:23geom}, several angle structures are possible after the 2-3 move as it can be seen in the poof of Lemma~\ref{lem:aux23geom}. Conversely, several angle structures lead to the same result for the 3-2 move. This is because angle structures are blind to shearing singularities. If such a singularity exists, a 3-2 move can still be performed while maintaining the whole angle structure, but it will decrease the hyperbolic volume of the structure. It is due to the fact that the singularity, which maximized the volume, is fixed in the process.
\end{remark}

\subsubsection{Computing the angles}

\begin{lemma}
	Let $\tri$ be a triangulation with angle structure, if a geometric Pachner move can be performed, then the new angle structure can be computed in a constant number of arithmetical operations.
	\label{lem:anglevalues}
\end{lemma}

\begin{proof}
	\noindent
	\textbf{3-2 move.\ } From the proof of Lemma~\ref{lem:32geom}, the new angles can be obtained by solving a system of equations with six equations.

	\noindent
	\textbf{2-3 move.\ } First, we remind that it is possible to convert angles to edge parameters \textit{et vice versa}~\cite{futer2010angled}. We still consider the notations of Figure~\ref{fig:pachlem}, but we consider the situation in the upper halfspace with $B$ sent to $0$, $E$ sent to $1$, and $D$ sent to $\infty$ (Figure~\ref{fig:anglecalc}). The position of $C$ depends on the edge parameter of $BD$ in $BCDE$, and the position of $A$ depends on the edge parameter of $BD$ in $ABCD$. Then the edge parameter of $BD$ in $ABDE$ is just the product of the two edge parameter. This construction is valid by Lemma~\ref{lem:23geom}, as the sum of the arguments of both shape parameters is smaller than $\pi$. We can use the same procedure to find an edge parameter for $ABCE$ and $ACDE$.
	
	Note that this construction leaves no shearing singularities around $AE$.
\end{proof}

\begin{figure}
	\centering
	\includegraphics[width = 0.5\columnwidth]{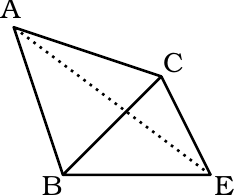}
	\caption{Representation of the proof of Lemma~\ref{lem:anglevalues}, seen from the point $D$. The 2-3 move will delete the line $BC$ in favor of $AE$.}
	\label{fig:anglecalc}
	
\end{figure}

\subsection{Getting rid of flat tetrahedra}

Let $\tri$ be a triangulation with an angle structure admitting a flat tetrahedron $t$, and $(e,e')$ the associated edges with dihedral angle $\pi$. By Lemma~\ref{lem:geom}, it is not possible to get rid of a tetrahedron with a geometric 2-3 move. Indeed, if an edge of $t$ has a dihedral angle equal to $\pi$, the second tetrahedron concerned by the 2-3 move must have a dihedral angle equal to 0, and consequently a 2-3 move will produce a flat tetrahedron. 

In order to get rid of $t$, our strategy is to turn either $e$ or $e'$ into an edge of degree three, \ie\ to put it at the center of three tetrahedra on which a 3-2 move can be performed (central edge of Figure~\ref{fig:pachlem}, right). 
Note that edges of degree 2 prevent the existence of an angle structure as they force the two tetrahedra sharing the edge to be flat. As a consequence, we assume these configurations are removed, which can be done by \snappy's simplification for instance, see Section~\ref{sec:expe}. Consequently, all edges have degree at least 3, and our strategy focuses on \emph{reducing} the incidence degree of angle $\pi$ edges.

\begin{figure}
	\centering
	\includegraphics[width= \columnwidth]{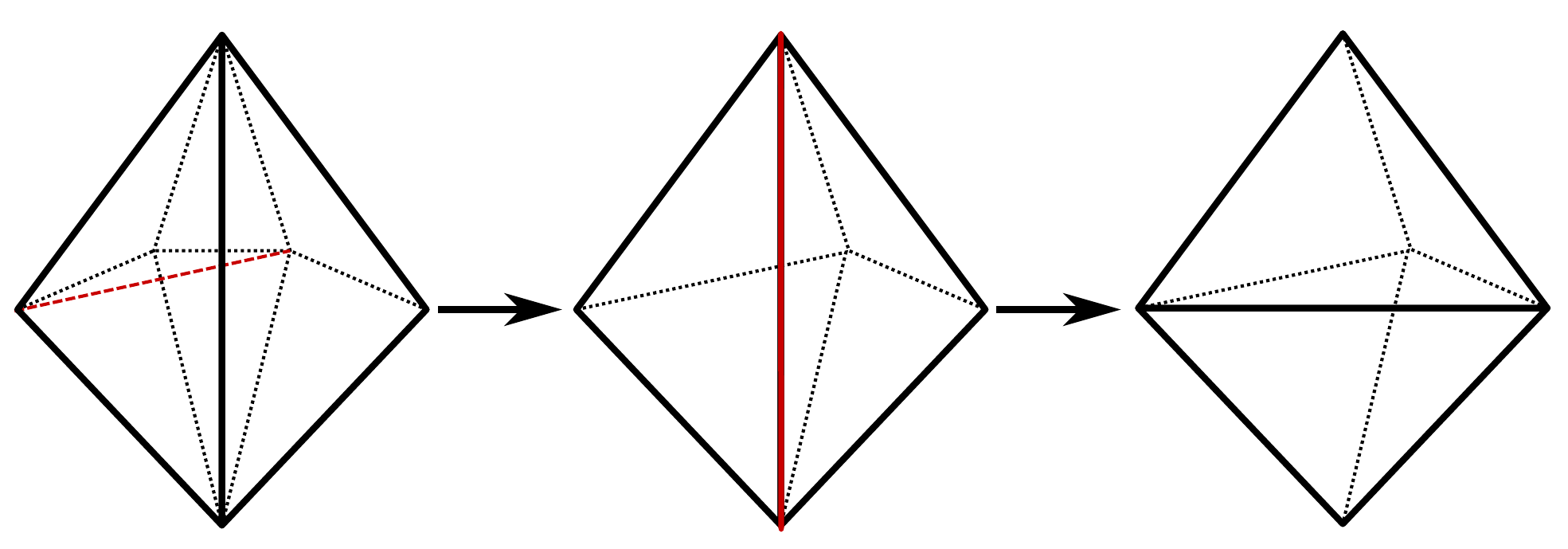}
	\caption{A sequence of moves getting rid of a flat tetrahedron. The bold vertical edge is contained in 4 tetrahedra, three of which are represented behind the edge. The fourth tetrahedron is implicit, situated at the front, and flat. %
		First, create the red edge of the first drawing with a 2-3 move. Then delete the red edge of the second drawing with a 3-2 move.}
	\label{fig:comb23}
\end{figure} 

\medskip
\noindent
{\bf To reduce the degree of an edge $e$,} the strategy is to perform 2-3 moves on the tetrahedra containing $e$, such that each move reduces the degree of the edge by one. Either for topological (tetrahedron glued to itself) or geometrical reasons, these moves will not always be possible, and the order in which they are done matters. This is described in Figure~\ref{fig:comb23}.

\begin{remark}
	Doing a 2-3 move does not always reduce the degree of the edge, \eg, when a tetrahedron is represented several times around an edge. However, doing a move that does not decrease the degree of the edge may delete the multiple occurrences of a tetrahedron around the edge and allow to continue with the simplification.  
\end{remark}

\begin{remark}
	When a 2-3 move is performed around $e$, the value of the dihedral angle of the new tetrahedron around $e$ is equal to the sum of the previous two dihedral angles around $e$. Since the sum of the dihedral angles of a tetrahedron is equal to $\pi$, this means successfully reducing the degree of $e$ may unlock previously forbidden moves. 
\end{remark}

\medskip
\noindent
{\bf Recursive moves.} If no geometric Pachner move is possible on an edge $e$, we attempt to reduce the degree of a nearby edge $e_f$, \ie, an edge for which there exists a tetrahedron containing $e$ and $e_f$; see Figure~\ref{fig:comb32} where $e_f$ is represented by the red dot, as it is seen from above. More precisely, let $f$ be a triangle containing $e$ such that the associated 2-3 move is forbidden for geometric reasons, then by Lemma~\ref{lem:geom} there is an edge $e_f$ of $f$ for which the associated dihedral angle is at least $\pi$. 

In consequence, reducing the degree of $e_f$ to three without modifying the tetrahedra containing $f$ and then performing a 3-2 move at $e_f$ will reduce the degree of $e$ by one; see Figure~\ref{fig:comb32}.

Note that in case a geometric Pachner move is not possible on any pair of tetrahedra containing $e_f$ either, we can call the procedure recursively in a neighborhood of $e_f$, using the argument above to pass from $e$ to $e_f$.

\begin{figure}
	\centering
	\includegraphics[width= \columnwidth]{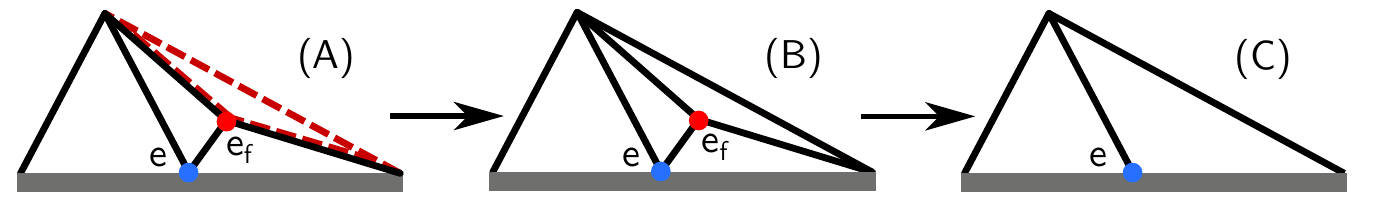}
	\caption{Sequences of moves to reduce the degree of an edge when no 2-3 move is available, seen from above (triangles represent sections of tetrahedra, edge section of triangles and dots section of edges). In all drawings, the gray rectangle represents the flat tetrahedron, $e$ (light blue dot) is the edge of the flat tetrahedron of which we want to reduce the degree. Assuming no 2-3 move is available, we pick another edge $e_f$ (red dot), if a 3-2 move can be performed on $e_f$, we are in case (B) and the 3-2 move reduces the degree of $e$ (case (C)). Otherwise, we are in case (A) we need to create the red dashed tetrahedron by decreasing the degree of $e_f$, this can be done by calling recursively our procedure on $e_f$.}
	\label{fig:comb32}
	
\end{figure} 

\begin{remark}
	The choice of $e_f$ in $f$ is unique: it is not possible to have two dihedral angles larger than $\pi$; selecting a move to reduce the degree of $e$ boils down to selecting a triangle containing $e$.
\end{remark}

\medskip

{\bf Procedure summary.} The procedure to get rid of a flat tetrahedron is a tree-like \emph{backtracking search}, each branch corresponding to a choice of 2-3 or 3-2 move to perform. The moves used are geometric to preserve the volume maximization advancement. A flat tetrahedron can only be deleted with a 3-2 move: the first step is to \emph{reduce the degree} of one of its edges to three. To reduce the degree of an edge, the solution is to perform 2-3 moves on two tetrahedra containing this edge. If a 2-3 move is not possible for geometric reasons, we are in the situation of Figure~\ref{fig:comb32}, and the procedure to \emph{reduce the degree} can be called \emph{recursively} on a neighboring edge. Performing a 3-2 move on this latter edge resulting in a decrease of the degree of the initial one. The procedure ends when the initial edge has degree three and the flat tetrahedron can be removed or when no move can be applied.

\subsection{Implementation details}

The implementation faces several practical challenges.

\medskip

\noindent
{\bf Breaking infinite loops.} 
As such, the algorithm may loop infinitely on some instance, as 2-3 and 3-2 moves may reverse themselves. To counteract this phenomenon and avoid redundant modifications, a solution is to store at each modification the \emph{isomorphism signature} of the triangulation~\cite{regina}, characterizing uniquely the isomorphism type of the triangulation. 
This allows us to recognize already processed triangulations and break branches of the backtracking algorithms. Note however that this is costly compared to the Pachner moves, and it makes the procedure no longer local. 

\medskip

\noindent
{\bf Selecting the edge $e$.} When attempting to remove a flat tetrahedron, one needs to choose between the two $\pi$-angled edges $e$ and $e'$ to reduce. 
In our implementation, we consider both edges, and select the one that has the larger smallest angle in its link. This performs better than a random choice in our experiments.

\medskip

\noindent
{\bf Pruning the backtrack search.} Some triangulations may have edges of large degree (more than 10), which produces wide search trees. Additionally, the recursive calls to the degree reduction procedure may induce trees of large depth. Experimentally, the better strategy consists of exploring exhaustively the first levels of the tree. We set the width of exploration to $8$ and explore the first $2$ levels of the search tree.

\medskip

\noindent
{\bf Sequencing of Pachner moves.} Different sequencings of Pachner moves lead to different triangula\-tions. In practice, we favor 2-3 moves over recursive ones, which performs best. Additionally, among the 2-3 moves performed to get rid of a flat tetrahedron containing edge $e$ of angle $\pi$, we prioritize moves that eliminate tetrahedra for which $e$ has a small dihedral angle. Among the recursive ones, we used the same method as the one choosing between $e$ and $e'$, indeed choosing a recursive move boils down to choosing an edge to minimize its degree.

\section{Experiments}
\label{sec:expe}
In this section, we study the experimental performance of our approach to find a triangulation with a \chs, and compare its behavior with the software \snappy{}~\cite{snappy}. \snappy{} is the state-of-the-art software to study the geometric properties of knots and 3-manifolds, and is widely used in the low-dimensional topology community. Following this analysis, we propose and study a hybrid method with practical interest.

\medskip

\noindent
{\bf \snappy{}.} \snappy{}'s method is based on an implementation from Weeks, its uses a random re-triangulation followed by a simplification. The first step is constituted of $4n$ random 2-3 moves, where $n$ is the number of tetrahedra in the input triangulation. The simplification performs non-deterministic modifications to decrease the number of tetrahedra in the triangulation. Notably it removes some configurations that prevent angle structures from existing. The verification of the existence of a \chs\ is based on a Newton's method to solve the gluing equations~\cite{weekscompletehyperbolic}.

\medskip

\noindent
{\bf Data set.} We apply our algorithms to the census of prime hyperbolic knots with up to 19 crossings. This census has been constituted by the efforts of many researchers in the field, and recently completed with the exhaustive enumeration of all knots with crossing number smaller than 20\footnote{The census is available at \url{https://regina-normal.github.io/data.html}}~\cite{burton350}.

For each knot, given by a knot diagram, we compute a triangulation of the knot complement using \regina{}~\cite{regina}, and simplify it with \snappy{}. We then keep the triangulations admitting an angle structure but not a \chs, and hence requiring re-triangulation. All the others, without angle structure or admitting a \chs, were discarded. They constitute the \emph{Failure on first try} data of Table~\ref{tab:reginacensus}.

The knots are grouped by crossing numbers (from 14 to 19), and on whether they are alternating or not. In Section~\ref{subsec:expetri}, our experiments are run on the first 2500 knots of each of the 12 groups having an angle structure but no \chs, and in Section~\ref{sec:hybrid} we use the first 10.000 knots of the same data sets.

\subsection{Our program and current state of the code}

Our implementation is done in \texttt{Python}, to be able to interface with \regina\ and \snappy\ (and not \texttt{SnaPpea}). The idea behind using \texttt{Python} was also to enable shorter development times. The basic tools to handle triangulations where rewritten to have more control on the behavior of the simulation and we rely on external libraries only for the isomorphism signatures, some linear programming to find a point to begin the optimization, and to perform the optimization (we use the \texttt{SLSQP}~\cite{SLSQP} optimization method from SciPy~\cite{SciPy}). 

However this choice of language leads to costly computation times, decision was taken to not render public this version of the code and switch to \texttt{C++}, using another library to manipulate the triangulations. 

The version of the program presented in Section~\ref{sec:hybrid}\footnote{Available at \url{https://github.com/orouille/SnapPy}}, which outperforms the state of the art in terms of performance, is a simpler \texttt{C} version directly implemented in \snappy's kernel. We are in discussion with the developers to integrate this code in the library \snappy.

\subsection{Success rate and combinatorial performance}
\label{subsec:expetri}

{\bf In terms of Pachner moves.} Figure~\ref{fig:expself} (left) represents the rate of knots on which our algorithm succeeds to find a triangulation with a \chs, for all the 12 groups of knots. It represents the success rate as a function of the number of Pachner moves, the higher success rate groups are the alternating knots. With Figure~\ref{fig:expself} (right), which represents this success rate for optimization phases, we see that few Pachner moves are done on average between the phases, hinting that the convex optimization will likely be the time bottleneck.

While there is a significant difference between the efficiency on alternating and non-alternating knots, all the curves have the same behavior: the first few  Pachner moves are very effective. An important point is that, compared to the number of moves done by \snappy{}, at least four times the number of tetrahedra and then roughly the same number for the simplification times the number of re-triangulation, our method uses a much lower number of Pachner moves on average (see Figure~\ref{fig:expself}, left). This is of \emph{combinatorial interest}, as doing a small number of moves can mitigate the impact of the procedure on the properties of the triangulation, such as keeping a small treewidth for instance.

After the first few steps, the growth of the success rate slows down drastically. An interpretation of this phenomenon is that the tetrahedra created by the 2-3 moves tend to to be more flat than the original ones. This leads to an increase in the number of required re-triangulations and issues with floating point arithmetic.

\begin{figure}
	\centering
	\includegraphics[width = 7cm]{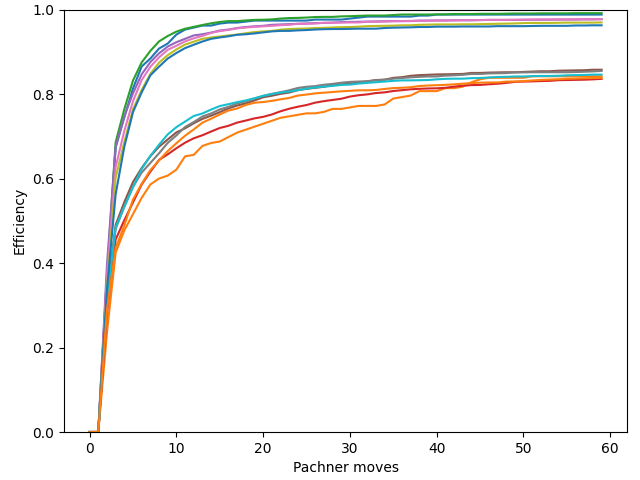} \hspace{0.5cm}
	\includegraphics[width = 7cm]{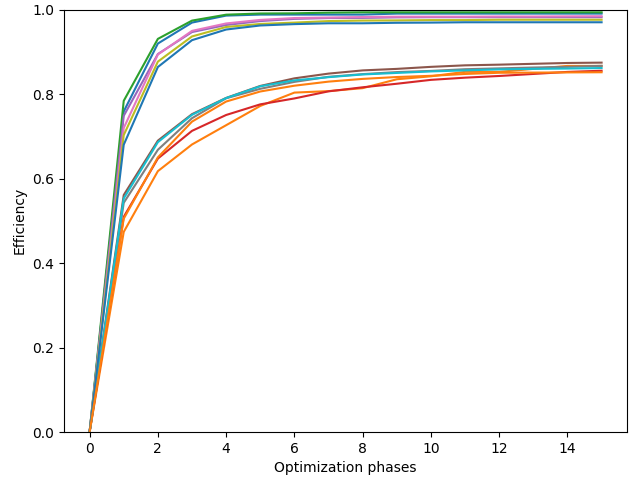}
	\caption{On the sample of triangulations, success rate in finding a \chs\ for different crossing numbers and alternabilities. The limit rate of success for alternating knots is 0.98, the limit rate for non-alternating knots is 0.87. Left: success rate against the number of Pachner moves. Right: success rate against the number of optimization phases.}
	\label{fig:expself}
	
\end{figure} 

\begin{figure}
	\centering
	\begin{tabular}{|r|c|c|c|c|c|c|c|c|c|c|c|c|}
		\hline
		& \multicolumn{5}{|c|}{\# of optim phases} & & \multicolumn{5}{|c|}{\# of moves} \\
		\hline
		& 1 & 2 & 3 & 4 & 5 &  & 2 & 3 & 4 & 10 & 30 \\
		\hline
		Efficiency alter & 0.73 & 0.89 & 0.95 & 0.97 & 0.97 & & 0.36 & 0.63 & 0.72 & 0.92 & 0.97\\  
		\hline
		Efficiency non-alter & 0.53 & 0.67 & 0.74 & 0.78 & 0.81 & & 0.30 & 0.47 & 0.52 & 0.70 & 0.81\\
		\hline
	\end{tabular}
	\caption{Average of success rates for the alternating and non-alternating groups of knots of Figure~\ref{fig:expself}.} 
	\label{tab:self}
\end{figure}

\medskip

\noindent
{\bf In terms of \emph{re-triangulation}.} In order to compare the behaviors of our method and of the state of the art: we want to know if the two methods struggle on the same triangulations. Both methods using sequences of re-triangulations, we use the length of these sequences as metrics.

We compare in Figure~\ref{fig:expcomparebase} the average number of randomized re-triangulations required by \snappy{} to find a triangulation with a \chs, and the (deterministic) number of re-triangulations of our method. The majority of the \chs\ are found within few re-triangulations on average: both methods manage to find 78\% of them within two steps. 

However, it appears that for the manifolds requiring a large number of re-triangulations with either method, the other one will perform well on it: the performance of the methods are substantially orthogonal for the difficult cases. This phenomenon highlights the fact that the methods are \emph{complementary}.

\begin{figure}
	\centering
	\includegraphics[width = 11cm]{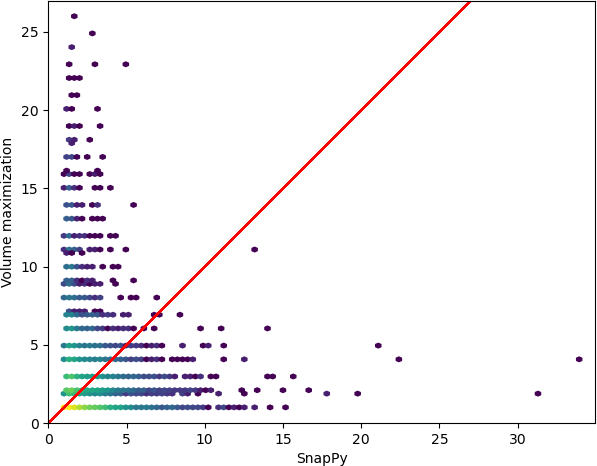}
	\caption{Number of re-triangulations required by our method over the number of re-triangulations required by \snappy{} on the 12 groups of knots, gathered in bins (the color indicates the log of the number of points inside). Since our method is deterministic, the points have discrete ordinates. The diagonal $x=y$ appears in red.}
	\label{fig:expcomparebase}
	
\end{figure} 

\medskip

\noindent
{\bf Limitations.} It is to be noted that our method suffers from several phenomena: there are manifolds on which our method fails, the reasons to this being convergence problems because of the floating point arithmetic, the creation of flatter tetrahedra or the volume modification mentioned in Remark~\ref{rem:23holo}. Furthermore some triangulations have several very high degree edges and looking for the correct sequence of geometric Pachner moves can be costly. In terms of time performance, the constraint convex optimization is an important bottleneck, as well as the computation and book-keeping of a large number of isomorphism signatures.

\medskip

However, in light of the previous results, we design a hybrid algorithm, mixing \snappy{}'s pipeline together with the heuristics based on localized geometric Pachner moves, in order to outperform both methods in terms of experimental timings.

\subsection{Hybrid algorithm and time performance}
\label{sec:hybrid}

As is, \snappy{} does not produce angle structures as defined in Section~\ref{sec:cassonrivin}, as it constructs \emph{negative tetrahedra}, i.e., tetrahedra whose shape has a negative imaginary part. Our hybrid method consists of calling \snappy{} to randomize and simplify the triangulation, and, in case the triangulation admits only few negative tetrahedra (less than four in our experiments), call the resolution of flat tetrahedra with geometric Pachner moves on these negative tetrahedra, as introduced in Section~\ref{sec:heuristic}. Loop until a \chs\ is found. 

Our implementation is done in {\tt C} directly in \snappy's kernel and consists of a simplified version of Section~\ref{sec:heuristic}'s method. For each negative tetrahedron, we try to reduce the degree of one of its edges down to 3 in turn using only geometric 2-3 moves, we try all the edges, and we do not use recursive moves or heuristics to perform the geometric moves.

The results of this method are summarized in Figure~\ref{fig:exphyb} (Figure~\ref{fig:exphybhist} is given to help read Figure~\ref{fig:exphyb}), where the time to compute the \chs\ are compared for the new hybrid method and the usual \snappy{} pipeline, on the first 10.000 triangulations of the datasets of Section~\ref{sec:expe}. 
The hybrid method using localized geometric Pachner moves shows, at worst, similar time performance compared to \snappy{}, and performs significantly better overall. On average, the hybrid method is 20\% faster, however, it is to be noted that, on most cases, few re-triangulations are required in which case the hybrid method and \snappy{} show naturally similar time performance. More interestingly, on cases requiring many more re-triangulations in the \snappy{} pipeline, the hybrid methods performs much better than \snappy{}, with running times up to 18 times faster.

As indicator of the global behavior of both algorithms, we indicate the linear regression of the data points (black dashed line, slope $= 0.4$) in Figure~\ref{fig:exphyb}, to illustrate that the dense region is not concentrated on the $x=y$ diagonal. Furthermore, this line suggests that if we are not in the dense region where almost all \chs\ are found instantly, then the gain is not 20\%, but 60\%.

\begin{figure}
	\centering
	\includegraphics[width = 13cm]{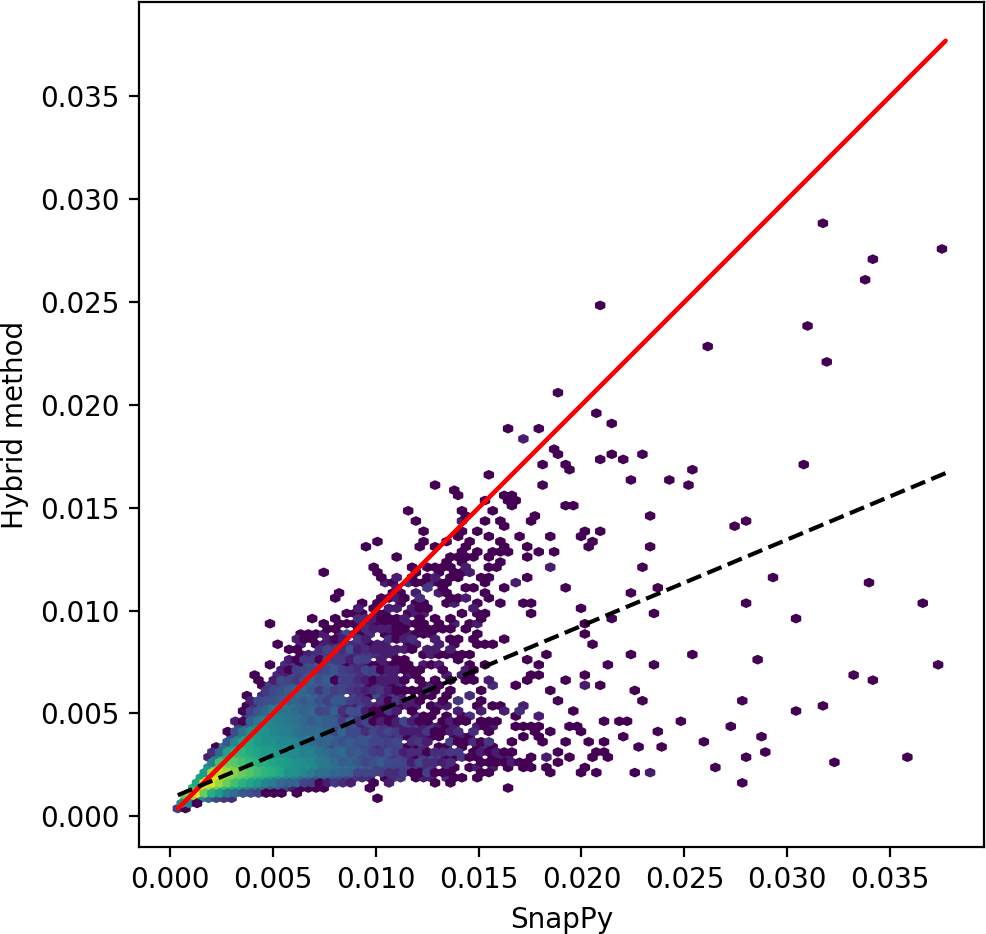}
	\caption{Time comparison of the hybrid method and \snappy{} in seconds for the 10.000 first triangulations of each crossing number of the data set of the previous section, gathered in bins (the color indicates the log of the number of points inside). Line $x=y$ is plain and red, a linear regression through the data set is dashed and black. For readability, four outlayers favoring the hybrid method are omitted. 
	}
	\label{fig:exphyb}
\end{figure}

\begin{figure}
	\centering
	\includegraphics[width = 9cm]{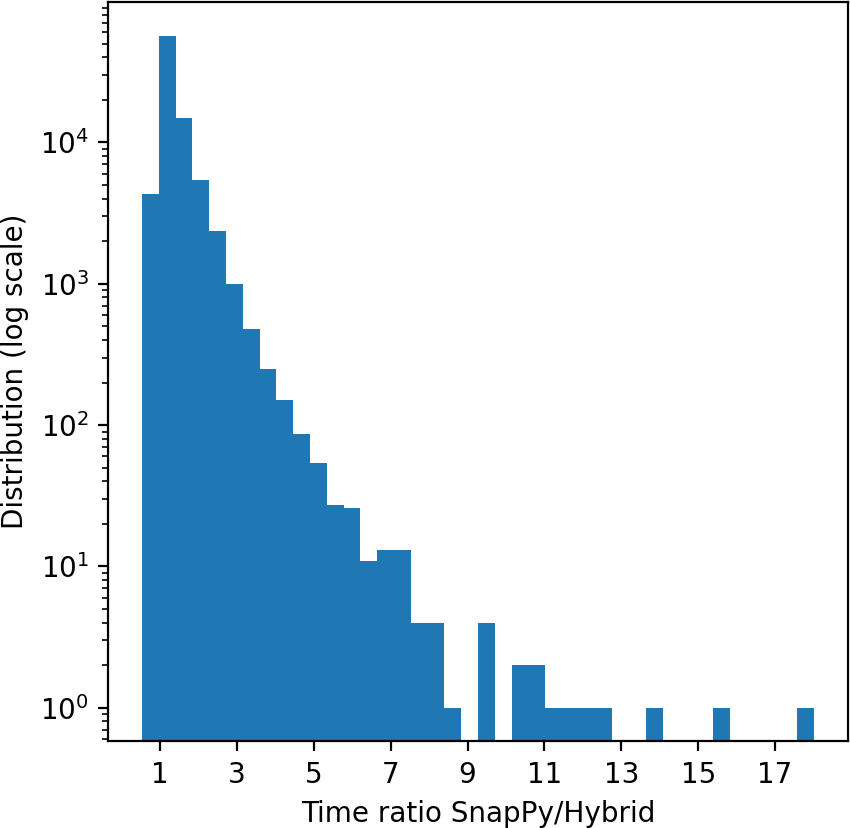}
	\caption{Distribution, in logarithmic scale, of the time required by \snappy{} over the time required by the hybrid method. 
	}
	\label{fig:exphybhist}
\end{figure} 

\section{Conclusion}
\label{sec:hypvolconcl}
The problematic of this work is the computation of complete hyperbolic structures on triangulated 3-manifolds. This can be done on computers by solving Thurston's gluing equations. The state-of-the-art method, implemented by \snappy, consists in trying to simplify the triangulation before using Newton's method to solve the equations. If this does not succeed, then the only strategy is to find another triangulation of the manifold and to start over. This work well in most cases, but can require a large number of re-triangulations in some cases. The key difficulty is to find a triangulation where the equations admits a solution

We have introduced a new method based on results form Casson and Rivin stating that maximizing the hyperbolic volume yields \chs. When it is not possible to find a solution to the equations on a triangulation, we use the result of the volume maximization to propose another one.

Our algorithm is complementary to \snappy{}, and in particular performs much better on almost all cases we have studied where \snappy{} struggled to find a solution to the \chs\ problem. While our current implementation of the whole method is not competitive due to being implemented in \texttt{Python}, we introduced a hybrid method that improves the state-of-the-art in most cases, specifically in difficult ones, without sacrificing \snappy's efficiency when it performs well. 

Our method proposes an improvement in term of computation time over the state-of-the-art, this is very important when running large scale experiments or when the manifolds studied are large. But our method has another interest: since it tries to avoid performing a large number of re-triangulations, we tend to preserve the triangulation whenever possible.

The method and its implementation can be improved: when used alone, it can simply fail, being able to continue with the information we gained would be a great plus. Also, we plan on transferring the full method in \texttt{C++}, to check its actual performances when \texttt{Python} is not used, this will rise questions about the convex programming and arithmetic precision.

The method opens the door for new questions concerning its behavior. A group of questions arises when looking at the volume maximization like a flow: What is the impact of the volume maximization on the triangulation? Which tetrahedra shrink first? Are there combinatorial configurations blocking our method? 
These questions are close to the study of the \emph{combinatorial Ricci flow}~\cite{xu2020combinatorial}, which is dual to our approach, meets similar problems and has known some development in the recent years.
Closely linked to the flow, there are questions on the path we take on path in the Pachner graph: these are quite popular in computational topology and geometry, with the problem to know what is happening at knocking nodes for instance. More generally, our graph is less connected than the topological one, how connected is it?  

There are other questions concerning our results: given the fact that the dataset we used is actually sorted by the volume, does the hyperbolic volume have an impact on the success rate of the methods? How does the approach perform on very large knots that are not on the census? Are the triangulations we output suitable for Weeks' program that computes canonical triangulations from complete hyperbolic structures~\cite{weekscanonical}?

There are also questions about \snappy's method we can look at with ours:  \snappy\ re-triangulation is a large number of 2-3 moves performed at random, why does it miss so much on some specific examples? What makes triangulations difficult to handle for \snappy\ and easy for us, \textit{et vise versa}?

In conclusion, while this work has a concrete and direct use in 3-manifolds study, it raises many interesting questions in mathematics, combinatorics, data science, and algorithmics.

\vfil

\pagebreak

\bibliography{bib}

\end{document}